\documentclass[12pt,english]{amsart}
\usepackage[T1]{fontenc}
\usepackage{textcomp}
\usepackage[latin9]{inputenc}
\usepackage{prettyref}
\usepackage{amstext}
\usepackage{amsthm}
\usepackage{amssymb}
\usepackage{amsmath,amssymb,amsthm,amsfonts,amsbsy,latexsym,dsfont}
\makeatletter
\numberwithin{figure}{section}
\usepackage{enumitem}
\theoremstyle{plain}
\newtheorem{thm}{\protect\theoremname}
\theoremstyle{plain}
\newtheorem{conjecture}{\protect\conjecturename}
\theoremstyle{plain}
\newtheorem{lem}[thm]{\protect\lemmaname}
\theoremstyle{plain}
\newtheorem{cor}[thm]{\protect\corollaryname}
\theoremstyle{plain}

\theoremstyle{plain}

\theoremstyle{definition}

\usepackage{fullpage}
\usepackage{amsfonts}
\usepackage{pdfsync}
\usepackage{tikz}
\tikzset{ invisnode/.style={circle, draw=white, fill=white, inner sep=0.01cm},
}
\usepackage{prettyref}
\usepackage{color}
\usepackage{comment}

\usepackage{cite}
\DeclareMathOperator*{\ra}{ran}
\DeclareMathOperator*{\dom}{dom}

\usepackage[b]{esvect}

\usepackage{fancyhdr}
\usepackage{color}

\usepackage{accents,graphicx}

\makeatletter
\DeclareRobustCommand{\cev}[1]{%
  \mathpalette\do@cev{#1}%
}
\newcommand{\do@cev}[2]{%
  \fix@cev{#1}{+}%
  \reflectbox{$\m@th#1\vec{\reflectbox{$\fix@cev{#1}{-}\m@th#1#2\fix@cev{#1}{+}$}}$}%
  \fix@cev{#1}{-}%
}
\newcommand{\fix@cev}[2]{%
  \ifx#1\displaystyle
    \mkern#23mu
  \else
    \ifx#1\textstyle
      \mkern#23mu
    \else
      \ifx#1\scriptstyle
        \mkern#22mu
      \else
        \mkern#22mu
      \fi
    \fi
  \fi
}

\makeatother

\usepackage{babel}
\providecommand{\conjecturename}{Conjecture}
\providecommand{\corollaryname}{Corollary}
\providecommand{\lemmaname}{Lemma}
\providecommand{\theoremname}{Theorem}
\providecommand{\claimname}{Claim}
\providecommand{\questionname}{Question}
\providecommand{\examplename}{Example}

\def\wt{\widetilde}

\def\ZX{\color{purple}}
\newtheorem{prop}[thm]{Proposition}

\newtheorem{example}{Example}

\def\COMMENT#1{}
\let\COMMENT=\footnote

\begin{document}
\title{An introduction to equitable DP  coloring of  graphs}

\author{H. A. Kierstead}
\address{Arizona State University,Tempe, AZ, USA}
\email{kierstead@asu.edu}

\author{Alexandr Kostochka}
\address{University of Illinois at Urbana--Champaign, Urbana, IL 61801}
\email{kostochk@illinois.edu}
\thanks{Research of second author is supported in part by   NSF RTG Grant DMS-1937241.}

\author{Zimu Xiang}
\address{University of Illinois at Urbana--Champaign, Urbana, IL 61801}
\email{ zimux2@illinois.edu}
\thanks{Research of third author is supported in part  by NSF RTG Grant DMS-1937241.}


\date{\today}

\begin{abstract}
 A proper $k$-coloring of vertices of an $n$-vertex  graph is \emph{equitable} if the size of every color class
 is $\lfloor n/k\rfloor$ or $\lceil n/k\rceil$.
  An extension of  it to list coloring requires only that   the size of every color class
 is at most $\lceil n/k\rceil$.
 Such colorings have interesting applications and have been actively studied recently.
 In this paper, we extend the notion of equitable coloring to the more general notion of equitable DP coloring and study properties of the new parameter.
\end{abstract}

\maketitle

\section{Introduction}

Let $G=(V,E)$ be a graph. We let $|G|=|V|$ and $\|G\|=|E|$.  
{The {\em neighborhood}, $N(x)=N_G(x)$, of a vertex $x$ in $G$ is the set of neighbors of $x$ in $G$,  the
{\em degree} of  $x$ in $G$ is  $d(x)=d_{G}(x):=|N_{G}(x)|$. 
Further,} $\Delta(G)$ is the maximum degree of $G$, and $\delta(G)$ is the minimum degree of $G$. For a positive integer  $n$, set $[n]:=\{1,\dots, n\}$. 

For $d\geq 1$, an ordering $\sigma=(v_1,\ldots,v_n)$ of the vertices of an $n$-vertex graph $G$ is $d$-{\em degenerate}, if for each $1\leq i\leq n$, $|N(v_i)\cap \{v_1,\ldots,v_{i-1}\}|\leq d$. We call a graph {\em $d$-degenerate} if it has a $d$-degenerate ordering.

The goals of this paper are (a) to introduce the notions of equitable DP-coloring and strongly equitable DP-coloring, (b) to show examples when such colorings behave differently from equitable list coloring and strongly equitable list coloring, (c) to prove the basic results on such colorings, and (d) to state  questions/problems on such colorings that seem interesting to us. 
We start from definitions and { basic results on various versions} 
of equitable coloring.

\subsection{{Ordinary and equitable colorability}}

A (proper) $k$-coloring of a graph $G=(V,E)$ is a mapping $f:V\to [k]$ such that for each
$i\in [k]$ the {\em color class} $f^{-1}(i)$ is an independent set in $G$.

An {\em equitable $k$-coloring}  of an $n$-vertex graph $G=(V,E)$ is a $k$-coloring $f$ of $G$ such that each color class $f^{-1}(i)$ has $\lfloor n/k\rfloor$ or $\lceil n/k\rceil$ vertices.

Equitable coloring and similar notions for graphs  have applications in construction of
timetables, mutual exclusion
scheduling problem, and round-a-clock scheduling; see, e.g.~\cite{BEPSW,HHK,IL,KW,Mey,SBG,Tu}. 
This concept is also  useful in studying extremal combinatorial and
probabilistic problems, see e.g.~\cite{AF,AY,JR,Pe,RR}.

One of the central results on equitable coloring of graphs is the 
 Hajnal--Szemer\' edi Theorem below.
 
 \begin{thm}[Hajnal and Szemer\' edi \cite{HSz}, 1970]
\label{thm:H-Sz}Every graph $G$ with $\Delta(G)<k$ is equitably $k$-colorable.
\end{thm}

Another line of research considers equitable $k$-coloring of sparse graphs $G$ with
$k<\Delta(G)$. 
Meyer~\cite{Mey} proved that every tree with maximum degree $\Delta$ has an equitable $k$-coloring
for $k = 1 + \lfloor \frac{\Delta}{2}\rfloor$.
 The extremal example is the star $K_{1,\Delta}$: since the color class of the  center vertex has size one, no other class can have size greater than two.
 Bollob\' as and Guy~\cite{BG83} gave for trees a bound that is not a function of maximum degree.

\begin{thm}[Bollob\' as and Guy~\cite{BG83}, 1983]\label{BoG}
    A tree $T$ is equitably $3$-colorable if $|T|\ge 3\Delta(T)-8$ or $|T|=3\Delta(T)-10.$
\end{thm}
This result was extended to { $k$-colorings for all $k\geq 2$} and to all forests by Chen and Lih~\cite{CL94} and Miyata,  Tokunaga and Kaneko~\cite{MTK}.
Given a graph $G$ and a vertex $v\in V(G)$, let $\alpha_v=\alpha_v(G)$ denote the size of a maximum independent set in $G$ containing $v$. If a graph $G$ has an equitable $k$-coloring, then by definition, $\alpha_v(G)\geq \lfloor \frac{n}{k}\rfloor$ for every $v\in V(G)$.
Chen and Lih~\cite{CL94} and  independently Miyata,  Tokunaga and Kaneko~\cite{MTK} proved (and Chang~\cite{Ch08} gave a shorter and nicer proof) that this necessary condition is sufficient for forests.

\begin{thm}[Chen and Lih~\cite{CL94}, 1994, Miyata,  Tokunaga and Kaneko~\cite{MTK}, 1994]\label{CLCh}
    For a forest $T$ of order $n$ and integer $k\ge 3$, $T$ is equitably $k$-colorable if and only if $\alpha_v\ge\lfloor \frac{n}{k}\rfloor$ for every vertex $v\in T$.
\end{thm}

\subsection{{ List, equitable list, and strongly equitable list colorability}} 

Let $\Gamma$ be a set of \emph{colors} and $G=(V,E)$ be a graph. A \emph{list assignment} for $G$ (from $\Gamma$) is a function $L:V\rightarrow 2^\Gamma$.
An \emph{$L$-coloring} $f$ is a proper coloring of $G$ such that $f(v)\in L(v)$ for all $v\in V$.
If $L(v)=\Gamma$ for all $v\in V$ then $L$ is \emph{plain}, 
and all ordinary colorings of $G$ with colors from  $\Gamma$ are  $L$-colorings. 
If {$|L(v)|= k$} for each $v\in V(G)$, then $L$ is a \emph{$k$-list assignment}. 

Kostochka, Pelsmajer and West~\cite{KPW} proposed the following list version of equitable colorability.
A coloring $f$ of $G$ is {\em $m$-bounded} if 
$|f^{-1}(\alpha)|\le m$  for all $\alpha\in \Gamma$. 
 Given $G$ and $k$ write $m:=m(G,k):=\lceil\frac{|G|}{k}\rceil$ for the size of a largest color class in an equitable $k$-coloring of $G$, and call such a color class { \emph{full}}.
Given a $k$-list
assignment $L$ for $G$, Kostochka et al. defined $G$ to be  \emph{equitably $L$-colorable} if it has an $m$-bounded $L$-coloring, and they defined
 $G$ to be \emph{equitably list $k$-colorable (EL $k$-colorable)} if
it is   equitably $L$-colorable for every $k$-list assignment $L$.

In contrast to ordinary equitable coloring, in the list setting it is possible that the sizes of two color classes must differ significantly, 
as it could be that some color  appears only 
 in the lists of   few vertices.
 On the other hand, the theory of equitable list colorability reflects the theory of equitable colorability. This justifies the similar terminology.
The following analog of Theorem~\ref{thm:H-Sz} was  conjectured in~\cite{KPW}.
\begin{conjecture}[Kostochka, Pelsmajer and West, Conjecture 1.1 \cite{KPW}, 2003]\label{conj:list-HS}
    Every graph $G$ is EL {$k$-colorable} for each $k> \Delta(G)$.
\end{conjecture}

Partial cases of the conjecture were proved in~\cite{KPW,MP,WL,KK13,YZ1,YZ2,Na12}, but the conjecture is still wide open.

An anomaly of the definition of EL colorability is that for a
{ plain} $k$-list $L$, { an $L$-coloring of a graph $G$ may be $m$-bounded, but not equitable.} 
For example, if $L$ is a plain $3$-list assignment then an ordinary $2$-coloring of $C_4$ is not an equitable coloring (one color class is empty, 
while two have size $\lceil\frac{|C_4|}{3}\rceil=2$), but it is  a $\lceil\frac{|C_4|}{3}\rceil$-bounded $L$-coloring. 
Recently, we  \cite{KKX} observed that this only occurs when there are too many \emph{large color classes}; 
a class is large if it has size exactly $\lfloor\frac{|G|}{k}\rfloor+1$.
Note that if $k$ divides $|G|$ then $m=\lfloor\frac{|G|}{k}\rfloor$ and an  $m$-bounded coloring of $G$ has { $k$ full color classes}, but no large color classes. 
Given a graph $G$ and integer $k$, a coloring of $G$ is  {strongly $m$-bounded} if it is $m$-bounded and there are at most $|G|\bmod k$ large color classes. 
A graph is \emph{strongly equitably $L$-colorable (SE $L$-colorable)} if it has a strongly $m$-bounded  $L$-coloring.  It is \emph{strongly equitably list $k$-colorable (SEL $k$-colorable)} if it is SE $L$-colorable for
all $k$-list assignments $L$.
If $L$ is plain then every strongly bounded coloring of $G$ is an equitable coloring of $G$, and  so every SEL $k$-colorable graph is equitably $k$-colorable. 

We feel that { SEL $k$-colorability}  is the right extension of equitable { colorability} to the list setting. It may seem like a small difference, since the notions agree when $|G|$ is divisible by $k$, but in induction arguments, divisibility by $k$ may not be preserved.

We think that Conjecture~\ref{conj:list-HS} holds also for SEL colorability.

\subsection{{ DP, equitable DP, and strongly equitable DP colorability}}

About 10 years ago Dvo\v r\' ak and Postle~\cite{2018DvPo} 
invented an interesting generalization of list coloring in order to attack an open problem on list colorings of planar graphs. 
They called it \emph{correspondence   coloring}, and we call it \emph{DP coloring} for short and to reflect their impact. 

Before defining DP-coloring we need some notation. Let $K_{s,t}$ denote the complete bipartite graph with $s$ vertices in one part and $t$ vertices in the other. To specify the vertices of $G:=K_{s,t}$ we write $K_{s,t}(X;Y)$, where $X$ and $Y$ are sets or lists of vertices, and we may drop the subscripts. 

Let $G=(V,E)$ be a graph with list assignment $L:=V\rightarrow2^{\Gamma}$. Set $\wt{L}(x):=\{x\}\times L(x)$. A \emph{matching
assignment} for the pair $(G,L)$ is a function $M$ on $E$ such that $M(uv)$ is
a matching of $K(\widetilde{L}(u);\widetilde{L}(v))$. { A graph $H$ is a \emph{cover} for $G$ if there are
$L$ and $M$, where $L$ is a list assignment for $G$ and $M$ is a matching assignment for $(G,L)$, such
that}
\[
V(H)=\bigcup_{v\in V} \widetilde{L}(v)
~\text{and}~E(H)=\bigcup_{uv\in E}M(uv).
\]
 For $k\in\mathbb{N}$, if $L$ is a $k$-list assignment then $H$ is a $k$-cover.

Suppose $H$ is a cover for $G$. Then this is witnessed by a unique list assignment $L$ and a unique matching assignment $M$. It is convenient to think of $H$ as both a graph and a function $H(\cdot)$ on $V\cup E$ with $H(v):=L(v)$ and $H(uv):=M(uv)$.  Now there is no
need to explicitly mention a list assignment $L$ or a matching assignment $M$ when working with $H$. 
For clarity, we refer to the vertices of $H$ as nodes. 
So a node is an ordered pair $(v,\gamma)$ where $v\in V$ and $\gamma\in L(v)$. 
We may indicate adjacency in the cover $H$ by writing $(x,\alpha)\sim(y,\beta)$.
It is also convenient to define the partial injection $H(u,v):L(u)\to L(v)$ by $H(u,v)(\alpha)=\beta$ if $(u,\alpha)\sim(v,\beta)$. Then $H(u,v)=H^{-1}(v,u)$.
{We emphasize that the statement $H(u,v)(\alpha)=\beta$ means  $uv\in E$ and node $(u,\alpha)$ is matched to node $(v,\beta)$ by  matching $M(uv)$. 
Then $H(u,v)(\alpha)\ne\beta$ means   $uv\notin E$ or  $(u,\alpha)(v,\beta)\notin M(uv)$. }

Given a cover $H$ of a graph $G$ and a $W\subset V(G)$, we denote by $H[W]$ the restriction of $H$ on $G[W]$.

An $\ensuremath{H}$\emph{-coloring} of $G$ is a function $f:V\rightarrow\Gamma$
with $f(v)\in H(v)$ such that $\{(v,f(v)):v\in V\}$ is independent in $H$. For brevity, sometimes instead of writing $f(v_1)=\alpha_1$ $,\ldots, f(v_t)=\alpha_t,$ we will write $f(v_1,\ldots,v_t)=(\alpha_1,\ldots,\alpha_t)$.
The
graph $G$ is \emph{DP $k$-colorable} if it has an $H$-coloring for every $k$-cover
$H$ of $G$. 

{ Call $H$ \emph{plain} if $H|V$ is plain. In this
case assume that $\Gamma=[|\Gamma|]$. Call $H$ \emph{normal} if $H(u,v)$ is the identity function for all $uv\in E$.
Every list coloring problem can be represented as a DP coloring problem by letting $H$ be normal. Similarly every ordinary coloring problem can be represented as a DP-problem by letting  $H$ be plain and normal.}

Many results on list coloring remain true for DP coloring. For example, every planar graph is DP $5$-colorable, every $d$-degenerate graph  is DP $(d+1)$-colorable, and a version of Brooks' Theorem holds for DP coloring.

A graph $G$ is \emph{DP degree-colorable} if $G$ has an $H$-coloring whenever $H$ is a cover of $G$ with $|H(u)|\geq d_G(u)$ for all $u\in V(G)$. 
A \emph{GDP-forest} is a graph in which every block is either a complete graph or a cycle.
A {\em GDP-tree} is a connected GDP-forest. 

A classical theorem of Gallai~\cite{gallai1963kritische}
on degree-colorable graphs extends  to DP coloring as follows. 

\begin{thm}[Dvo\v r\' ak and Postle~\cite{2018DvPo}]
\label{thm:deg-choosable}
    Suppose that $G$ is a connected graph. Then $G$ is not DP degree-colorable if and only if $G$ is a GDP-tree.
\end{thm}

As we already have the definitions of equitable and strongly equitable list { colorability, it is natural to define equitable and strongly equitable DP colorability analogously.} 
Given a $k$-cover $H$ of $G$, 
call $G$ \emph{equitably $H$-colorable} if $G$ has an { $m$-bounded} $H$-coloring, and call $G$ \emph{equitably DP  $k$-colorable (EDP $k$-colorable)} if it is equitably $H$-colorable
for all $k$-covers $H$ of $G$. 
Call $G$ \emph{strongly equitably  $H$-colorable (SE $H$-colorable)} if there is a strongly $m$-bounded $H$-coloring of $G$, and call $G$ \emph{strongly equitably DP $k$-colorable (SEDP $k$-colorable)} if 
  $G$ is SE $H$-colorable for every $k$-cover graph $H$ of $G$. 

  \subsection{Terminology}
  {One issue remains. What should we call graph  colorings that witnesses that a graph is equitably $L$-colorable, strongly equitably $H$-colorable, etc? Kostochka et al. used the term \emph{$m$-bounded  L-coloring.} But they also informally used the undefined term \emph{equitable $L$-coloring}. Since then the latter formulation has been used by many authors including us. It is not really correct; grammatically, it should refer to an $L$-coloring that is equitable in the ordinary sense. This issue becomes more noticable as we explore new versions of equitable colorability. Here we will revert to \emph{$m$-bounded $L$-coloring} and use \emph{strongly $m$-bounded $H$-coloring, etc}.}

\subsection{Organization}
While EDP coloring is similar to EL coloring, the situation with bounds is very  different. For example,   {$K_{2}$} is not  EDP 2-colorable, a graph on $n$ vertices  may not be EDP $n$-colorable, and so on.
In this paper we present some introductory results on EDP and SEDP coloring. 


In the next section, we 
  present a number of examples of graphs $G$ such that 
  are SEL $k$-colorable, but not SEDP $k$-colorable.
 In Section 3 we prove 
 that if $G$ is not SEDP $(n+d)$-colorable then $|G|>n$ or $G$ is not $d$-degenerate, and that there are infinitely many $d$-degenerate graphs on $n$ vertices that are not EDP $(n+d-1)$-colorable.
  In Section 4 we consider SEDP colorability of forests with given maximum degree.  In Section 5 we show that if
 $G$ is not EDP $k$-colorable then $k<3\Delta^2(G)-2$. We think that this bound is far from optimal.
 In Section 6 we conclude with some comments and questions.
 
 \subsection{Notation.} 
 
 Let $G$ and $H$ be graphs. Set $G\vee H=G\cup H \cup K(V(G),V(H))$. Let $\overline{G}$ be the complement of $G$. Let $G+H$ be the union of \emph{disjoint copies} of $G$ and $H$.

 Let $K_{n}$ be the complete  graph on $n$ vertices. Let $K(X)$ be the complete graph with vertex set $X$ or $\{X\}$ depending on whether $X$ is a set or a list. Similarly, $K(X,Y):=\overline{K}(X)\vee \overline{K}(Y)$. Let $P_{n}$ be the path on $n$ vertices and $C_{n}$ be the cycle on $n$ vertices. Let $P=v_{1}\dots v_{n}$ be the path with vertices $v_{1 },\dots,v_{n}$ whose edges are $v_{1}v_{2},\dots, v_{n-1}v_{n}$. Let $v_{1}\dots v_{n}v_{1}$ be the cycle $C=P+v_{n}v_{1}$.

 Given a graph $G$, sets $A,B\subseteq V(G)$ and a vertex $x\in V(G)$, we denote by $E(A,B)$ the set of edges with one end in $A$ and the other in $B$, and let $||A,B||:=|E(A,B)|$, $A+x:=A\cup \{x\}$, $A-x:=A\smallsetminus \{x\}$ and $N[x]:=N(x)\cup \{x\}$.

 For a function $f$ and a set $A$, we denote by $\dom(f)$---the domain of $f$, by $\ra(f)$---the range of $f$, by $f|A$---the restriction of $f$ to $A$, and
by $2^A$---the set of all subsets of $A$.

 Addition { and subtraction modulo $s$ are denoted by $\oplus_s$ and $\ominus_s$, respectively,} and we drop $s$ when it is clear from the context.
 
Let $S_n$ denote the set of permutations of $[n]$.  We use cycle notation to express such permutations. For example, 
\[S_3=\{(1~2~3),(1
~3~2),(1)(2)(3),(1)(2~3),(1~3)(2),(1~2)(3)\}.\] 

 \section{Examples}
 In this section we provide some examples to illustrate complications of  EL and EDP { colorability}. To start,
 observe that if a graph is DP { $k$-colorable} then all its subgraphs are too. But an EL { $k$-colorable} graph may have subgraphs that are not EL { $k$-colorable}. For instance, consider $K_{1,3}\subseteq K_{1,3}+K_{1,3}$ with $k=2$. 

Every path is DP { $2$-colorable} and SE{L} $2$-colorable. However:
\begin{example}[Paths] No path $P_n=v_1\dots v_n$ with $n\ge2$ is 
 EDP $2$-colorable.
\end{example}

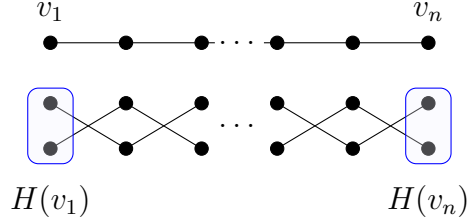
\begin{figure}
    \centering
    \begin{tikzpicture}[scale =1]
\def \vt{circle (2.5pt) [fill]} 

\foreach \j in {0,1,2,3,4,5} { 
    \draw (\j,0) \vt;  
    
} 
\node[invisnode] at (2.2,0) (a) {};
\node[invisnode] at (2.8,0) (b) {};
\foreach \j in {0,1,3,4} { \draw (\j,0)--(\j+1,0); } 
\draw (2,0)--(a);
\draw (b)--(3,0);
\node[] at (2.5,0)  {$\dots$};

\node[label={above:$v_1$}] at (0,0) (v1) {};
\node[label={above:$v_n$}] at (5,0) (vn) {};

\foreach \j in {0,1,2,3,4,5} { 
    \draw (\j,-.8) \vt;  
    \draw (\j,-1.4) \vt; 
} 

\foreach \j in {0,1,3,4} {\draw (\j,-.8)--(\j+1,-1.4);\draw (\j,-1.4)--(\j+1,-.8);
}

\node[] at (2.5,-1.1)  {$\dots$};

\draw[rounded corners, color=blue, fill=blue!5,fill opacity=0.3](-.3,-1.6) rectangle (.3,-.6);
\node[label={below:$H(v_1)$}] at (0,-1.6) (Hv1) {};
\draw[rounded corners, color=blue, fill=blue!5,fill opacity=0.3](4.7,-1.6) rectangle (5.3,-.6);
\node[label={below:$H(v_n)$}] at (5,-1.6) (Hv6) {};

\end{tikzpicture} 
    \caption{Paths are not EDP $2$-colorable.}
    \label{fig:path}
\end{figure}

\begin{proof}See Figure~\ref{fig:path}. Suppose $H$ is a plain cover for { $P_n$}. For all 
$i\in[n-1]$, let {$H(v_i,v_{i+1}):=(1~2)$.} Then all $H$-colorings $f$ satisfy $f(v_1)=\ldots=f(v_n)$, and so are not equitable.
\end{proof}
We will show in Lemma~\ref{path} that  $P_n$ is SEDP $k$-colorable  for every $k\geq 3$ and  $n\neq 3$.

\medskip
No cycle is DP $2$-colorable. All cycles are DP  and SEL $k$-colorable for all $k\ge3$. But:
\begin{example}[Cycles] For all cycles $C_n= v_1\dots v_nv_1$: 
\begin{enumerate}[label=(\alph{enumi})]
    \item  $C_3$  is neither EDP $3$-colorable nor EDP $4$-colorable;
    \item  $C_4$  is neither SEDP $3$-colorable nor EDP $4$-colorable; and    
    \item  $C_6$ is not EDP $3$-colorable.
\end{enumerate}
\end{example}

\begin{proof} For (a), suppose $n=3$. 
Let $H_3$ be the plain $3$-cover for $C_3$ with {$H_3(v_i,v_{i\oplus 1}):=(1~2)(3)$} for all $i\in[3]$. 
Any $H_3$-coloring of $C_3$ must only use two colors, and so is not { $\lceil\frac{|C_3|}{3}\rceil$-bounded}. 
Similarly, let $H_4$ be the plain $4$-cover for $C_3$ such that {$H_4(v_i,v_{i\oplus1}):=(1~2)(3~4)$} for all $i\in [3]$. Any $H_4$-coloring must only use two colors, and so is not { { $\lceil\frac{|C_3|}{4}\rceil$-bounded}}.

For (b), suppose $n=4$. { Let $H_3$ be the plain $3$-cover such that 
\[H_3(v_i,v_{i\oplus1}):=
\begin{cases}
(1~2~3)&\text{if $i$ is odd}\\
(1~3~2)&\text{if $i$ is even}.
\end{cases}\]
See Figure~\ref{fig:cycle}. Suppose $f$ is an SE $H_3$-coloring of $C_4$. Then some color $\alpha$ is used twice, and the other two colors are used once. By symmetry, assume $f(v_1)=1$, $|f^{-1}{\ZX (1)}|=2$, and $f(v_3)\ne1$. 
Then $f(v_4)=3$, and so $f(v_2)\notin \{2,3\}$. Thus $f(v_3)=1$ and $f(v_2)=3$, a contradiction.} 

Now let $H_4$ be the plain $4$-cover for $C_4$ such that,
{\[H_4(v_i,v_{i\oplus1}):=
\begin{cases}
(1~2~3)(4)&\text{if $i$ is odd}\\
(1~3~2)(4)&\text{if $i$ is even.}
\end{cases}\]}Suppose $f$ is a { $\lceil\frac{|C_4|}{4}\rceil$-bounded} $H_4$-coloring of $C_4$. It must use all four colors. 
 Say $f(v_4)=4$. Then $f$ is an { $\lceil\frac{|C_4|}{3}\rceil$-bounded} $H$-coloring of $v_1v_2v_3$, where $H$ is the plain $3$-cover of the path $v_1v_2v_3$ with $H(v_1,v_2):=(1~2~3)$ and $H(v_2,v_3):=(1~3~2)$. Say $f(v_2)=\alpha$. 
 Then $\alpha\ominus1\notin\{f(v_1),f(v_3)\}$, a contradiction.

For (c), let $n=6$. Let $H$ be the plain $3$-cover for $C_6$ such that,
{\[H(v_i,v_{i\oplus1}):=
\begin{cases}
(1~2~3)&\text{if $i\in\{1,3\}$}\\
(1~3~2)&\text{if $i\in \{2,4\}$}\\
(1)(2)(3)&\text{if $i\in \{5,6\}$}.
\end{cases}\]}Suppose $f$ is a { $\lceil\frac{|C_6|}{3}\rceil$-bounded} $H$-coloring of $C_6$. It must use each color twice. Say $f(v_3)=\alpha$. Then $f(v_2)\ne \alpha\oplus1$ and $f(v_4)\ne \alpha\oplus1$. There is $i\in\{2,4\}$ with $f(v_i)\ne \alpha$. Say $f(v_4)\ne\alpha$.Then $f(v_4)=\alpha\ominus1$, and so $f(v_5)\ne\alpha\oplus1$. Thus $f(v_1)=\alpha\oplus1=f(v_6)$, a contradiction. 
\end{proof}

\medskip
We think that we can prove for every $n\geq 7$ the cycle $C_n$ is EDP $3$-colorable.

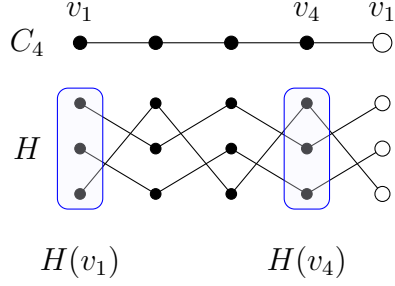
\begin{figure}
    \centering
    \begin{tikzpicture}[scale =1]
\def \vt{circle (2.5pt) [fill]} 

\foreach \j in {0,1,2,3} {
    \draw (\j,0) \vt;
}
\foreach \j in {0,1,2,3} { \draw (\j,0)--(\j+1,0); } 
\draw (4,0) node[circle, draw, inner sep=2.5pt, fill=white] {};
\node[label={above:$v_1$}] at (0,0) (v1) {};
\node[label={above:$v_4$}] at (3,0) (v4) {};
\node[label={above:$v_1$}] at (4,0) (v1') {};

\node at (-0.7, 0) {$C_4$};

\def \svt{circle (2pt) [fill]}
\foreach \j in {0,1,2,3} { 
    \draw (\j,-.8) \svt;  
    \draw (\j,-1.4) \svt; 
    \draw (\j, -2) \svt;
}

\foreach \j in {0,2} {\draw (\j,-.8)--(\j+1,-1.4);\draw (\j,-1.4)--(\j+1,-2);\draw (\j,-2)--(\j+1,-.8);
}
\foreach \j in {1,3} {\draw (\j,-.8)--(\j+1,-2);\draw (\j,-1.4)--(\j+1,-.8);\draw (\j,-2)--(\j+1,-1.4);
}

\draw (4,-.8) node[circle, draw, inner sep=2pt,fill=white] {};
\draw (4,-1.4) node[circle, draw, inner sep=2pt,fill=white] {};
\draw (4,-2) node[circle, draw, inner sep=2pt,fill=white] {};
\node at (-0.7, -1.4) {$H$};

\draw[rounded corners, color=blue, fill=blue!5,fill opacity=0.3](-.3,-2.2) rectangle (.3,-.6);
\node[label={below:$H(v_1)$}] at (0,-2.4) (Hv1) {};

\draw[rounded corners, color=blue, fill=blue!5,fill opacity=0.3](2.7,-2.2) rectangle (3.3,-.6);
\node[label={below:$H(v_4)$}] at (3,-2.4) (Hv4) {};

\end{tikzpicture} 
    \caption{$C_4$ is not SEDP $3$-colorable witnessed by the cover $H$.}
    \label{fig:cycle}
\end{figure}

Every forest $G$ is DP $2$-colorable and equitably $(\lceil\frac{\Delta(G)}{2}\rceil+1)$-colorable. By Theorem~\ref{BoG}, every $n$-vertex forest with maximum degree at most $\frac{n+8}{3}$ is equitably $3$-colorable.
However:
\begin{example}[Forests] \label{ex3}~
\begin{enumerate}[label=(\alph{enumi})]
\item 
The tree $T:=u_1uu_2\cup v_1vv_2\cup uv$ (a balanced double star) is not EDP $3$-colorable;
\item for all $k\geq 1$, the forest $F_k:=K(u;u_1,\dots u_{k+1})+\sum_{i\in[k-1]}K(x_i;y_i)$ consisting of the sum of a star on $k+2$ vertices and $k-1$ disjoint edges is not EDP $3$-colorable.
\end{enumerate}
\end{example}

\begin{figure}
    \begin{tikzpicture}[scale =1] 
\def \vt{circle (2.5pt) [fill]} 
\draw (0,0) \vt;
\node[label={above:$u$}] at (0,0) (u) {};
\draw (1,0) \vt;
\node[label={above:$v$}] at (1,0) (v) {};
\draw (-.7,.7) \vt;
\node[label={left:$u_1$}] at (-.7,.7) (u1) {};
\draw (-.7,-.7) \vt;
\node[label={left:$u_2$}] at (-.7,-.7) (u2) {};
\draw (1.7,.7) \vt;
\node[label={right:$v_1$}] at (1.7,.7) (v1) {};
\draw (1.7,-.7) \vt;
\node[label={right:$v_2$}] at (1.7,-.7) (v2) {};
\draw (-.7,-.7)--(0,0)--(1,0)--(1.7,-.7);
\draw (-.7,.7)--(0,0);
\draw (1,0)--(1.7,.7);
\end{tikzpicture}
\hspace{2cm}
\begin{tikzpicture}[scale =1]
\def \vt{circle (2.5pt) [fill]}
\draw (0,0) \vt;
\node[label={above:$u$}] at (0,0) (u) {};
\draw (-.9,-1) \vt;
\node[label={below:$u_1$}] at (-.9,-1) (u1) {};
\draw (-.3,-1) \vt;
\node[label={below:$u_2$}] at (-.3,-1) (u2) {};
\draw (.3,-1) \vt;
\node[label={below:$u_3$}] at (.3,-1) (u3) {};
\draw (.9,-1) \vt;
\node[label={below:$u_4$}] at (.9,-1) (u4) {};

\draw (2,0) \vt;
\node[label={above:$x_1$}] at (2,0) (x1) {};
\draw (3,0) \vt;
\node[label={above:$x_2$}] at (3,0) (x2) {};
\draw (2,-1) \vt;
\node[label={below:$y_1$}] at (2,-1) (y1) {};
\draw (3,-1) \vt;
\node[label={below:$y_2$}] at (3,-1) (y2) {};
\draw (-.9,-1)--(0,0)--(-.3,-1);
\draw (.9,-1)--(0,0)--(.3,-1);
\draw (2,0)--(2,-1);
\draw (3,0)--(3,-1);
\end{tikzpicture} 
\caption{Forests $T$ (left) and $F_3$ (right) are not EDP $3$-colorable.}
\label{fig:forests}
\end{figure}
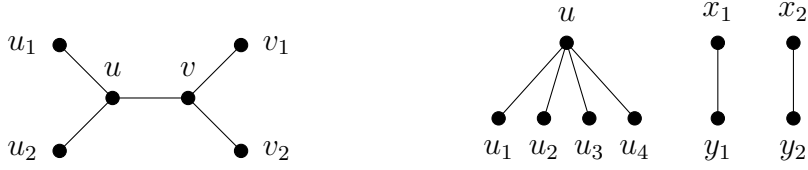

\begin{proof} 
See Figure~\ref{fig:forests}. For (a),  let $H$ be the plain $3$-cover of $T$ with $H(u,u_1)=H(u_2,u)=H(v,v_1)=H(v,v_2)=H(u,v)=(1~2~3)$.
Suppose $f$ is an equitable $H$-coloring of $T$. Then each color class has size two.
By symmetry, assume $f(v)=1$. Then $f(x)\neq 2$ for $x\in\{v_1,v_2,u\}$, so $f(u_1)=f(u_2)=2$. Thus $f(u)\ne3$ and $f(u)\ne1$, a contradiction.

For (b), let  $H$ be the plain $3$-cover of $F_k$ with $H(u,u_i)=(1~2~3)$ for all $i\in [k+1]$ and $H(x_i,y_i)=(1)(2)(3)$ for all $i\in [k-1]$.
Suppose $f$ is an equitable $H$-coloring of $C$.
Then each color class has size $k$. By symmetry, assume $f(u)=1$. Then $f(u_i)\ne 2$ for $i\in[k+1]$. Thus there is $i\in [k-1]$ with $f(x_i)=2=f(y_i)$,  a contradiction.
\end{proof}

Example~\ref{ex3} shows that Theorem~\ref{BoG} does not directly extend to equitable DP  coloring. Theorem~\ref{t13} in Section~\ref{sec4} proves that for each $D\geq 4$ every
forest with maximum degree at most $D$ apart from the star $K_{1,D}$ is SEDP-$k$-colorable for every $k\geq D$.

\medskip
For every $k\geq n$ every $n$-vertex graph is SEL  $k$-choosable and DP  $k$-colorable. However:

{ \begin{example}[On $n$-vertex, $d$-degenerate graphs]\label{ex:4} 
The $n$-vertex, 
$d$-degenerate graph $G_{n,d}:=K(D)\vee\overline{K}(S)$, where $|D|=d$ and $|S|=n-d$, is not EDP $k$-colorable for 
\[n\leq k\leq n+d-1-((n+d-1)\bmod 2).\] 
{In particular, 
if $n\leq k\leq 2n-2$, then  $K_n=G_{n,n-1}$ is not EDP $k$-colorable; if in addition, $n=2s+1$ then $K_{1,2s}=G_{n,1}$ also is not EDP $n$-colorable.}
 \label{exkn}
\end{example}

\begin{proof} 
Let $H$ be the plain $k$-cover of $G_{n,d}$ such that for all edges $uv\in G_{n,d}$, 
\[H(u,v):=
\begin{cases}
 (1~2)(3~4)\dots(k-1~~k)& \text{if $k$ is even}\\
 (1~2)(3~4)\dots(k-2~~k-1)(k)& \text{if $k$ is odd.}    
\end{cases}
\]
{ See Figure~\ref{fig:ex4}.} Suppose $f$ is an equitable $H$-coloring of $G_{n,d}$. 
As $k\ge n$, every color class has size at most one, and no vertex in $S$
can be colored with any color in  \[\{f(v):v\in D\}\cup \{H(u,v)(f(v)):v\in D\}.\] 
If $H(u,v)(f(v))=f(v)$ then $f(v)=k$ and $k$ is odd. Thus after coloring $D$, there are at most $k-2d+(k\bmod2)\le n-d-1$ colors left for the  $n-d$ vertices in $S$, a contradiction. 
\end{proof}}

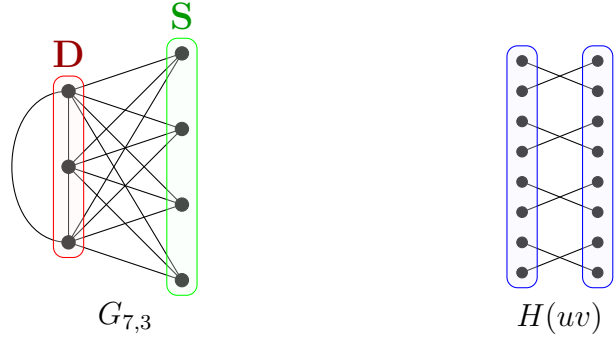
\begin{figure}
    \centering
\begin{tikzpicture}[scale=1]
\def \vt{circle (2.5pt) [fill]} 
\foreach \j in {-1.5,-.5,.5} { \draw (0,\j) \vt; }
\draw (0,-1.5)--(0,-.5)--(0,.5);
\draw (0,-1.5) .. controls (-1,-1.5) and (-1,.5) .. (0,.5);
\foreach \j in {-2,-1,0,1} {\draw (1.5,\j) \vt;}
\foreach \i in {-1.5,-.5,.5} {
        \foreach \j in {-2,-1,0,1} {
            \draw (0,\i) -- (1.5,\j);
        }
    }

\draw[rounded corners, color=red, fill=red!5,fill opacity=0.3](-.2,.7) rectangle (.2,-1.7);
\node[red!60!black] at (0,1) {\large\textbf{D}};

\draw[rounded corners, color=green, fill=green!5,fill opacity=0.3](1.3,-2.2) rectangle (1.7,1.2);
\node[green!60!black] at (1.5,1.5) {\large\textbf{S}};
\node[label={below:$G_{7,3}$}] at (.75,-2) (G73) {};

\begin{scope}[xshift=3cm, yshift=-.3cm]
\def \vt{circle (2pt) [fill]} 

\foreach \j in {0,1,2,3,4,5,6,7} {
    \draw (3,1.2-.4*\j) \vt;
    \draw (4,1.2-.4*\j) \vt;
}
\foreach \j in {0,2,4,6} {
    \draw (3,1.2-.4*\j)--(4,.8-.4*\j);
    \draw (3,.8-.4*\j)--(4,1.2-.4*\j);
}
\draw[rounded corners, color=blue, fill=blue!5,fill opacity=0.3](2.8,1.4) rectangle (3.2,-1.8);
\node[label={below:$H(uv)$}] at (3.5,-1.7) (Huv) {};
\draw[rounded corners, color=blue, fill=blue!5,fill opacity=0.3](3.8,1.4) rectangle (4.2,-1.8);
\end{scope}

\end{tikzpicture}
    \caption{$G_{7,3}=K(D)\vee\overline K(S)$ is not equitably  $H$-colorable, where $H$ is the $8$-cover for $G_{7,3}$ with
    $H(u,v)=(1~2)(3~4)(5~6)(7~8)$ for all edges $uv\in G_{7,3}$.}
    \label{fig:ex4}
\end{figure}
Proposition~\ref{Easy} below implies that for all $n>d\geq 0$ every
$n$-vertex $d$-degenerate graph is SEDP-$k$-colorable for every $k\geq n+d$.

\medskip

For all positive integers $n$ and all $k\geq n+1$, the complete bipartite graph $K_{n,n}$ is DP $k$-colorable and equitably $k$-colorable. It is conjectured that it is also SEL  $k$-colorable.
But:
\begin{example}[Balanced complete bipartite graphs]\label{ex:5}   $G:=K_{n,n}$ is not  EDP $2n$-colorable. 
\end{example}

\begin{proof} 
Let $G:=K_{n,n}(U;W)$.
 Let $H$ be the plain $2n$-cover for $G$ with $H(u,w)=(1~2~\dots~2n)$ for all edges $uw$ { with $u\in U$}.
Suppose $G$ has a { $1$-bounded} $H$-coloring $f$. 
Then every color class has size one. 
By symmetry, assume there is $u\in U$ with $f(u)=1$. Let $\beta$ be the least color such that $f^{-1}(\beta)\cap W\ne\emptyset$. Then no vertex is colored with $\beta-1$, a contradiction.
\end{proof}

By Example~\ref{ex:5} there are graphs $G$ with $|G|=2n$ and $\Delta(G)=n$ that are not  EDP $2n$-colorable. 
For every $k\geq 1$ every graph $G$ with $\Delta(G)\leq k$ is DP $(k+1)$-colorable. By Hajnal-Szemer\'edi Theorem, every such graph is equitably $(k+1)$-colorable. 
It is conjectured that every such graph is also SEL  $(k+1)$-colorable. However: 

\begin{example}[Graphs with given maximum degree] 
For all positive integers $k$, there are infinitely many graphs $G$ with 
$\Delta(G)\leq k$ 
that are not
 EDP $(k+1)$-colorable.
\end{example}

\begin{proof} 
Let $G_s=\sum_{i=0}^sF_i$, where each $F_i=K_{\Delta+1}$.
Let $H$ be a plain $(\Delta+1)$-cover of $G_s$ such that {$H[G_i]$} is normal for all $i\in [s]$, but  {$H[F_0]$} is as in Example~\ref{exkn}.
\end{proof}
In Section~\ref{sec5} we show that every graph $G$ is SEDP $k$-colorable for every
$k\geq 3(\Delta(G))^2$.

\medskip
{ 
For any $s\geq 1$, the star $K_{1,2s}$  is SEL  $(s+1)$-colorable and DP $(s+1)$-colorable.
But:
\begin{example}[Stars]\label{ex7}  The star $G_s:=K(\{x\},L)$  is not  SEDP $k$-colorable if $k\le n:=|G_s|$.
\end{example}

\begin{proof} 
Let $H$ be the plain $k$-cover for $G_s$ with $H(x):=[k]$ and $H(x,l):=(1~2~\dots~k)$ for all $l\in L$. 
Suppose  $f$ is an $H$-coloring of $G_s$.
Then $f(l)\ne \alpha:= f(x)\oplus1$  for all $l\in L$. Thus the remaining $k-1$ colors must be used on $n$ vertices. So some color $\beta$ is used at least twice. Since $f$ is plain, $\alpha$ and $\beta$ witness that $f$ is not strongly equitable. 
\end{proof}
}

 A partial case of Proposition~\ref{Easy} below yields that the star $K_{1,s-1}$
 is SEDP-$k$-colorable for every $k\geq s+1$.

 \section{Upper bounds for $d$-degenerate graphs}
 
 { In this section we study the SEDP $k$-colorability of $n$-vertex graphs, where $k$ is usually at least $n$.
 In the next two sections, while arguing inductively, we often apply these techniques to the last $k$ vertices of a graph.} 

Fix a graph $G=(V,E)$ and a $k$-cover $H$ for $G$ with $n:=|G|\le k$.
Now:
\begin{equation}\label{nlek}
\parbox{14cm}{\center{If $n\le k$ then an $H$-coloring $f$ of $G$ is SE if and only if  it is injective.}} 
\end{equation}
Suppose $f$ is an injective $H$-coloring of $G_0\subset G$. We can try to iteratively extend $f$ to an SE $H$-coloring of $G$. For this, consider a vertex $v\in G-G_0$ and  color sets:
\begin{enumerate} 
    \item $B(v)=B^f(v):=\{H(w,v)(f(w)):w\in N(v)\cap G_0\}$, 
    \item $R=R^f:=\ra(f)$, and  
    \item $A(v)=A^f(v):=H(v)\smallsetminus (B^f(v)\cup R^f)$.
\end{enumerate}
Say that the colors in $B(v)$ are \emph{blocked at $v$}, the colors in $R$ are \emph{used}, and the colors in $A(v)$ are \emph{available for $v$}.  
A color $\alpha\in B(v)$ is \emph{blocked by $u\in \dom(f)$} if $H(u,v)(f(u))=\alpha$. 
Now $f+(v,\alpha)$ is an SE $H$ coloring of $G[G_0+v]$ if and only if $\alpha \in A^f(v)$. 
When $G-G_0$ is independent, it suffices to find a system of distinct representatives (sdr) for $\{A^f(v):v\in V(G- G_0)\}$.
Thus, by Hall's Theorem, it suffices to check
\begin{equation}\label{IND}
\parbox{15cm}{\center{$G-G_0$ is independent and $|S|\le |\bigcup_{w\in S}A^f(w)|$  for all $S\subseteq V(G-G_0)$.
}}
\end{equation}

When $G-G_0$ is not independent we must work harder. 
We often do our iterative construction based on a degenerate ordering of $V$. 
Consider an ordering $\sigma=(v_1,\dots,v_n)$ of $V$ and a coloring $f$ of $G$.  
Set $G_i:=G[v_1,\dots,v_i]$. 
We intend that $f_i$ is $f$ restricted to $G_i$, but we may use this notation before $f$ is fully defined. So when we complete the construction, we will have $f=\bigcup f_i$. 
When $f_i$ is defined, 
 we shorten our notation to $A_i:=A^{f_{i-1}}(v_i)$, $B_i:=B^{f_{i-1}}(v_i)$, and $R_i:=R^{f_{i-1}}$. 

The $\sigma$-First-Fit algorithm ($\sigma$-FF) attempts to find an injective $H$-coloring $f$ of $G$ by iteratively picking $f_i(v_i)\in A_i$ in the order $\sigma$. The algorithm succeeds for  some sequence of choices $f_i\in A_i$ if and only if $f$ is an SE $H$-coloring of $G$. In this case, $f$ is called a $\sigma$-FF coloring. We say that $G$ is $\sigma$-FF SE $H$-colorable if $\sigma$-FF succeeds for all such choices.

  \begin{prop}\label{Easy} Suppose $0\leq d< n$, $k\geq n+d$,  $G$ is an $n$-vertex  graph, and
  $\sigma=(v_1,\ldots,v_n)$ is a $d$-degenerate ordering of $G$. Then $G$ is $\sigma$-FF SE $H$-colorable. 
  
\end{prop}
\begin{proof}
    Arguing by induction on $i\in [n]$, it suffices to show  $|A_i|>k-(i-1)-d$. Since each node $v_j\in N(v_i)$ with $j\in[i-1]$ blocks at most one color in $B(v_i)$, this follows from 
    \[|A_i|=|H(v_i)|- |R_i|-|B_i|\ge (n+d)-(i-1)-d\ge 1.\qedhere\]
\end{proof}

Example~\ref{ex:4} presents a graph $K_d\vee \overline K_{n-d}$ showing that this bound is sharp for all
$d\in [n-1]$ with $n+d-1$ even. We will prove a Brooks-type result, showing that there are no other examples when $k=n+d-1$ and $d>1$.
But first we need a definition.

For an edge $uv\in E$, $H(u,v)$ is a \emph{derangement} if {$H(u)\subseteq H(v)$}, $H(u,v)$ is defined on $H(u)$, and  $H(u,v)(\alpha)\ne\alpha$ for all $\alpha\in H(u)$.
A cover $H$ is \emph{constant} if it is plain and $H(u,v)=H(u',v')$ for all $uv,u'v'\in E$. 
Note that if $H$ is constant then $H(u,v)=H(v,u)=H^{-1}(u,v)$, so $H(u,v)$ has order $2$. 
If, in addition, $H(u,v)$ is a derangement, then $H(u,v)=(\alpha_1~\alpha_2)\dots(\alpha_{k-1}~\alpha_k)$ for some ordering $\alpha_1,\dots,\alpha_k$ of $H(u)$. 
In particular, $k$ is even.

\begin{thm}\label{n+d-1} Suppose $1\leq d< n$, $k= n+d-1$,
and $G=(V,E)$ is an $n$-vertex $d$-degenerate graph  with a $k$-cover $H$. If $G$ is not
SE $H$-colorable, then each of the following holds:
\begin{enumerate}[label=(\alph{enumi})]
    \item  $G=K_d\vee \overline K_{n-d}$;

\item  $H$ is plain and $H(u,v)$  is a derangement for all $uv\in E$;

\item for all $x\in V$ and $y,z\in N(x)$, $H(y,x)=H(z,x)$;

\item  if $d\geq 2$ then $H$ is { constant (so with (b), $k$ is even).
}
\end{enumerate}
\end{thm}

\begin{proof} 
Suppose the theorem fails for a  $d$-degenerate ordering $\sigma=(v_1,\ldots,v_n)$ of $G$.
By Proposition~\ref{Easy}, $G_{n-1}$ is $\sigma$-FF SE $H$-colorable. We will obtain a contradiction by constructing a $\sigma$-FF $H$-coloring $f'=f_{n-1}$  of $G_{n-1}$ with $A_n\ne \emptyset.$ Then
\begin{equation}\label{eq:Ai}
    |A_{i}|=|H(v_{i})\smallsetminus (R_{i}\cup B_{i})|\ge(n+d-1)-(i-1)-d=n-i.
\end{equation}
As $\sigma$ is $d$-degenerate, $d(v_n)\leq d$. If $d(v_n)<d$, then  \eqref{eq:Ai} overestimates $|B_n|$ by $1$, so $|A_n|\ge 1$, a contradiction. Thus
\begin{equation}\label{d=d}
    d(v_n)=d.
\end{equation}
We will derive (a--d) in a series of claims. One of our techniques will be to modify $\sigma$ in a way that preserves $d$-degeneracy. Note that 
\begin{equation}\label{eq:shift}
 \parbox{14cm}{{\center\em if $d(v_i)\le d$ then moving $v_i$ to a higher position does not increase degeneracy above $d$; neither does permuting $\{v_1,\dots,v_{d+1}\}$.}}   
\end{equation}
Another technique will be to place additional constraints on early $A_i$ to improve later $A_j$.

First we show that
\begin{equation}\label{eq:plain}
\parbox{14cm}{\center{ \em $H$ is plain.
}}
\end{equation}
Suppose not. Let $i\in [n-1]$ be the minimum index for which there is a color $\alpha\in H(v_i)\smallsetminus H(v_n)$.
Pick $f_{i-1}$ so that for all $j<i$, $f_j(v_j)\in A_j-H(v_i,v_j)(\alpha)$. 
This is possible by \eqref{eq:Ai}.
As $\alpha\notin \bigcup H(v_j)$, $\alpha\notin R_i$. 
By the construction, $\alpha\notin B_i$. 
As $\alpha\in H(v_i)$,  it is in $A_i$.
Pick $f(v_i)=\alpha$, and extend $f(v_i)$ by $\sigma$-FF to $f'$.
As $\alpha\in R_n\smallsetminus H(v_n)$, $A_n\ne\emptyset$, a contradiction.
 This proves \eqref{eq:plain}. 
 
 Let $L$ be the common list of all vertices in $G$.
Next, we show  that (b) holds for edges $v_iv_n$:
\begin{equation}\label{alp}
\parbox{14cm}{ for all $v_i\in N(v_n)$ and $\alpha\in L$, $H(v_i,v_n)(\alpha)$ is defined and 
$H(v_i,v_n)(\alpha)\ne \alpha$.
}
\end{equation}

Suppose not. Pick  $v_i\in N(v_n)$ and $\alpha\in L$ with
 $H(v_i,v_n)(\alpha)=\alpha$ or  undefined.  
We will obtain a contradiction by coloring $v_i$ with $\alpha$---then $R_n\cap B_n\ne \emptyset$ or $|B_n|<d$, so  $A_n\ne\emptyset$ by \eqref{eq:Ai}. 
For this, we show that (*) $A_j\smallsetminus \{\alpha,H(v_i,v_j)(\alpha)\}\ne \emptyset$, if $j<i$.
By \eqref{eq:Ai}, (*) holds unless:
\begin{equation}\label{eq:alp+}
    \parbox{14cm}{ $i=n-1$, $j=n-2$, $d_{G_j}(v_j)=d$, and  $H(v_iv_j)(\alpha)=f(v_j)=:\beta$ (so $v_iv_j\in E$).}
\end{equation}
By degeneracy and \eqref{eq:alp+}, there is $j'\le d<j$ with $v_{j'}\notin N(v_i)$. 
By \eqref{eq:shift}, the order of the first $d$ vertices does not affect $d$-degeneracy, so assume $j'=1$. 
Using \eqref{eq:plain}, let $f(v_1)=\beta$. 
Now $H(v_i,v_j)(\alpha)=\beta\notin A_j$. 
As $|A_j|\ge 2$,  (*) holds. 
This proves \eqref{alp}.

Next we prove a special case of (c).
 \begin{equation}\label{Alp'2}
\parbox{14cm}{\center If $v_i\in N(v_n)$ and $j\in[n]-i$, then  $v_jv_i\in E$ and
$H(v_j,v_i)=H(v_n,v_i)$.
}
\end{equation}
Suppose not.  
Consider the least $j$ with $v_j\notin N(v_i)$ or $H(v_j,v_i)(\beta)\ne H(v_n,v_i)(\beta):= \alpha$.
It suffices to construct $f'$   
so that $f(v_i)=\alpha$ and $f(v_j)=\beta$: Then $\beta\in R_n\cap B_n$, so $A_n\ne \emptyset$.
Now $\{v_{j'}:j'< j\}\subseteq N(v_i)$. By degeneracy, if $j<i$ then $j-1\le d$; else $i<j$ and  $i\le d+1$. 

 Suppose $i<j$. By \eqref{eq:shift} and $i\le d+1$,   assume $i=1$. 
Set $f(v_1):=\alpha$. 
Let $1<j'<j$.
As $H(v_1,v_{j'})(\alpha)=\beta$, $f(v_{j'})\ne \beta$, so 
 $\beta\notin R_j$. 
As $j'<j<n$,  $j'\le n-2$, so $|A_{j'}|\ge2$.
Iteratively, pick $f(v_{j'})\in A_{j'}-H(v_j,v_{j'})(\beta)$. 
Now $\beta\notin B_j$. 
So $\beta \in A_j$; setting $f(v_j)=\beta$, we are done.

Otherwise, $j<i$. By \eqref{eq:shift} and $j\le d+1$,  assume $j=1$. 
Set $f(v_1):=\beta$. 
Let $1<i'<i$. 
If $i\le n-2$, then $|A_{i'}|\ge3$. 
Iteratively color $G_{i-1}$ so that $f(v_{i'})\in A_{i'}\smallsetminus \{\alpha,H(v_i,v_{i'})(\alpha)$\}. Setting $f(v_i)=\alpha$, we are done.

Else, $i=n-1$. By the previous paragraph, it suffices to find a $d$-degenerate permutation of $(v_1,\dots,v_{n-1})$ that moves $v_{n-1}$. 
If $n\le d+2$ then move $v_i$ to the first position.
Else, if $v_{i'}\in K:=N(v_n)-v_i$ then $N[v_{i'}]=V$, 
so $K\subseteq\{v_1,\dots,v_{d+1}\}$; assume $K=\{v_1, \dots,v_{d-1}\}$. 
Set $X:=G-K$. Now, $X$ is $1$-degenerate and $v_iv_n\in E(X)$.
Thus $X$ is acyclic and has at least two leaves; one, say $v$, is not $v_n$. 
By \eqref{eq:shift} and $d_G(v)=|K|+d_X(v)=d$, we can move $v$ to the position between $v_i$ and $v_n$ in $\sigma$. Now $v_i$ is in the $n-2$ position. This proves \eqref{Alp'2}.

\medskip
Set $N:=N(v_n)$ and $Y:=V\smallsetminus N$. 
By \eqref{Alp'2},  $N[v]=V$  for all $v\in N$;
 by \eqref{d=d},  $|N|= d$. 
As $G$ is $d$-degenerate, $\omega(G)\le d+1$, so $Y$ is independent. 
Now $G=K(N)\vee \overline{K}(Y)$.
This proves (a). 
Now  $N\subseteq \{v_1,\dots,v_{d+1}\}$  because $\sigma$ is a $d$-degenerate ordering of $V$.

Suppose (b) fails. By \eqref{eq:plain}, $H$ is plain. Thus there are $uv\in E$ and $\alpha\in L$ with $H(u,v)(\alpha)=\alpha$ or  $H(u,v)(\alpha)$ undefined.  
By \eqref{alp}, $u,v\in N$. Assume $u=v_1$, $v=v_2$, and $H(u,v_n)(\alpha)=\beta$; 
by \eqref{alp}, $\beta$ exists and is not $\alpha$. 
Set $f(v_1)=\alpha$ and $f(v_2)=\beta$. 
Now $\beta\in R_n\cap B_n$, so $A_n\ne\emptyset$, a contradiction. This proves (b).

For (c), consider $x\in V$ and $y,z\in N(x)$. Suppose  $x\in N$. Assume $x=v_1$. By \eqref{Alp'2},
 $H(y,x)=H(v_n,x)=H(z,x)$.
Else, $x\in Y$, and $y,z\in N$. Assume $x=v_n$, $y=v_1$, and $z=v_2$.
Suppose $\alpha:=H(x,y)(\beta)\ne H(x,z)(\beta)=:\gamma$. 
If $H(y,z)(\alpha)\ne \gamma$ then set $f(y):=\alpha$ and $f(z):=\gamma$. 
Now $|B_n|\le d-1$, since $\beta$ is blocked at $x=v_n$ by both $y$ and $z$, so $f$ can be extended to $G$. 
Else $H(y,z)(\alpha)=\gamma$. 
Set $f(y):=\alpha$ and $f(z):=\beta$. Now $\beta\in R_n\cap B_n$, so $f$ can be extended to $G$. 
Anyway, we have a contradiction.  This proves (c).

{ For (d), suppose $d\ge 2$, and consider an edge $xy$. We must show that $H(x,y)=H(y,x)$. By (a), $xy$ is contained in a triangle $xyzx$. Using (c),  \[H(x,y)=_{\textrm (c)}H(z,y)=H^{-1}(y,z)=_{\textrm (c)}H^{-1}(x,z)=H(z,x)=_{\textrm (c)}H(y,x).\qedhere\] }

\end{proof}

{ Recall that forests are  $1$-degenerate graphs. In this case, 
Theorem~\ref{n+d-1} holds for fewer colors, if we only require an EDP coloring.}

\begin{cor}\label{n-1} Let  $k\geq \frac{n+2}{2}$.
If an $n$-vertex forest $G$ with a $k$-cover $H$ is not 
equitably $H$-colorable, then each of the following holds:
\begin{enumerate}[label=(\alph{enumi})]
\item  $k=n$ and $G=K_{1,n-1}~(=K_1\vee \overline{K}_{n-1})$;
\item  $H$ is plain and $H(u,v)$  is a derangement for all $uv\in E$;
\item for all $x\in V$ and $y,z\in N(x)$, $H(y,x)=H(z,x)$.
\end{enumerate}
\end{cor} 
{ 
\begin{proof}
    If $k\geq n$, all the claims follow from Proposition~\ref{Easy} or Theorem~\ref{n+d-1}, so
assume $\frac{n+1}{2}< k<n$. Let $\sigma=(v_1,\dots,v_n)$ be a $1$-degenerate ordering of $G$. Let $\{V_1,V_2\}$ partition $V(G)$, where $V_1:=\{v_1,\dots, v_t\}, t:=\lceil\frac{n}{2}\rceil$. By Proposition~\ref{Easy}, for each $i\in [2]$, there is a  $\sigma$-FF $H$-coloring $g_i$ of $G[V_i]$. Since $\sigma$ is $1$-degenerate, when constructing $g_2(v_i)$, we can avoid the unique color blocked by any previous $v_j$, even if $j\le t$. Now $g_1\cup g_2$ is an equitable $H$-coloring of $G$.
\end{proof}}

The proof of Corollary~\ref{n-1} does not necessarily provide an SE $H$-coloring because it does not control the number of color classes of size $2$. 
 By Example~\ref{ex7},   this problem is unavoidable. Below we prove a slight sharpening of Theorem~\ref{n+d-1} for stars.

\begin{prop}\label{star}
Suppose $G=(V,E)$ is the star $K(\{r\},X)$, $|G|=n\ge2$, and $H'$ is an $n$-cover for $G$. Let $u\in V$,    $\Gamma:=\bigcup_{v\in V-u}H'(v)$, $\gamma\in H'(u)$, and $H:=H'-(u,\gamma)$.
If $G$ has no SE  $H$-coloring, then
\begin{equation}\label{*}\tag{*}
\parbox{15.5cm}{
\begin{enumerate}[label=(\alph{enumi})]
    \item $H(u)\subseteq\Gamma$ and $|\Gamma|=n$, i.e., $H(u)\subseteq \Gamma$ and $H(w)=\Gamma$ for all $w\in V-u$; 
    \item $H(v,z)$ is a derangement for all $vz\in E$ with $v\in \{u,r\}$ and $z\ne u$; and
    \item $H(r,x)= H(r,y)$ for all  $x,y\in X-u$.
\end{enumerate}}
\end{equation}

\end{prop}  

\begin{proof} 

{
Arguing by contraposition, we assume (*) fails and find an SE $H$-coloring  of $G$. Our main tool is to color $r$ and then apply \eqref{IND}; for the analysis we may also color another vertex.

\emph{Case 1:}  (a) does not hold.  Then there are vertices $v,y\in V$ and a color $\beta\in H(v)$ with $y\ne u$ and $|H(y)-\beta|\ge n$. Prefer $v=u$, then $v=r$, then  $y\in X$. }
Set $f(v):=\beta$. 

Suppose $v=u$. If $u=r$ then $|A(y)|\ge n-1$ and $|A(x)|\ge n-2$ for all $x\in X$, so \eqref{IND} holds. 
Else $u\in X$. Now $A(r)\ne \emptyset$ since $r=y$ or $n\ge3$.  
If $y\ne r$ or $n=2$ then pick $f(r)\in A(r)$. 
In the latter case, $f$ is an SE $H$-coloring of $G$. In the former case,
 $|A(y)|\ge n-2$ and $|A(x)|\ge n-3$ for all $x\in X-u$, so \eqref{IND} holds.  
Else $y=r$ and $n\ge3$. Let $z\in X-u$.   As $z\ne y$, the preference for $y\in X$ gives  $|H(z)-\beta|\le n-1$. Thus there is $\alpha\in H(r)\smallsetminus H(z)$. Set $f(r):=\alpha$. Now $|A(z)|\ge n-2$ and $|A(x)|\ge n-3$ for all $x\in X-u$, so \eqref{IND} holds. 

Otherwise $v\ne u$. 
Now $v\ne u \ne y\ne v$, so $n\ge 3$. By the preference for $v=u$, $\beta\notin H(u)$. 
If $v=r$ then $A(u)\ne \emptyset$, $|A(y)|=n-1$, and $|A(x)|\ge n-2$ for all $x\in X-u$, so \eqref{IND} holds. 
Else, $v\in X-u$. 
By the preference for $v=r$, $\beta\notin H(r)$. If $r\ne u$ then $H(r)\smallsetminus H(v)\ne \emptyset$, contradicting $v\in X$. 
So $r=u$, and $A(r)\ne \emptyset$. 
Pick $f(r)\in A(r)$.  
Now $|A(y)|\ge n-2$ and $|A(x)|\ge n-3$ for all $x\in X$,  so \eqref{IND} holds. 
This proves (a).

\emph{Case 2:}  (b) does not hold. Then there is $vz\in E$ with $v\in \{u,r\}$ and $z\ne u$, and there is $\alpha\in H(v)$ such that either $H(v,z)(\alpha)=\alpha$ or it is undefined. If $r=u$ then $v=r$. Set $f(r):=\alpha$. Now
 $|A(z)|\ge n-1$ and $|A(y)|\ge n-2$ for every $y\in X$, so  \eqref{IND} holds. 
Else, $r\ne u=v$ and $z=r$.  Let $\gamma'\in \Gamma\smallsetminus H(u)$.
If there is $x\in X-u$ with $H(r,x)(\alpha)\ne\gamma'$ then set $f(r)=\alpha$. Now $\gamma'\in A(x)$, $\gamma\notin A(u)$, and $|A(y)|\ge n-2$  for all $y\in X$, so we are done by \eqref{IND}. Else put $f(r):=\beta$, where $H(r,u)(\beta)$ is undefined. Now $\gamma'\in A(x)\smallsetminus A(u)$ and $|A(y)|\ge n-2$ for all $y\in X$, so \eqref{IND} holds. This proves (b).

\emph{Case 3:}  (c) does not hold.  Then there are $x,y\in X-u$ and $\alpha \in H(r)$ with $H(r,x)(\alpha)\ne H(r,y)(\alpha)$. Set $f(r):=\alpha$. If $u\in X$ then $n\ge4$, so  $|A(u)|\ge 1$. Anyway,  $|A(z)|\ge n-2$ for all $z\in X-u$, and either $|A(x)|\ge n-1$ or $|A(y)|\ge n-1$, so \eqref{IND} holds. This proves (c).
\end{proof}

\section{Forests with given maximum degree}\label{sec4}

{
In this section we study 
EDP and}
 SEDP $k$-colorings of forests in terms of their maximum degree.
Call  a vertex $v\in G$ with $d(v)\le1$ an \emph{end}.  A vertex  $u$ is a \emph{hub} with \emph{body} $U$ and \emph{tail} $u'$ if 
{$\emptyset \ne U\subseteq N(u)\cap V_1(G)$ and   $N(u)= U+u'$.}

\begin{lem}\label{NS}If $G$ is a forest with $|G|\ge5$ that is not a star, then $G$ has two nonadjacent ends $l_1$ and $l_2$ such that  $G-l_i$ is not a star  for all $i\in[2]$. 
\end{lem}

\begin{lem}\label{hub}
    If $G$ is a { tree} with $\Delta(G)\ge2$  that is not a star, then $G$ has a hub.
\end{lem}

{

\begin{thm}\label{t13}
    Every forest $G$ is  a star or is SEDP $k$-colorable for all $k\ge \max\{\Delta(G),4\}$.
\end{thm}

 \begin{proof} Suppose not. Let  $G=(V,E)$ be a minimum counterexample. Then $G$ is a forest, but not a star, and there is a $k$-cover $H$ for $G$ such that $G$ has no SE $H$-coloring. 
By Proposition~\ref{Easy} and Theorem~\ref{n+d-1}(a),
$|G|> k$.

 Suppose $i:=|G| \bmod k>0$. Now $|G|\ge k+1\ge5$. 
 By Lemma~\ref{NS}, there is an end $l$ such that  $G':=G-l$ is not a star. So $G'$ has an SE $H$-coloring $f$.
  As $f$ has at most $i-1$ large classes, at least $k-(i-1)\ge2$ classes are not full. As  $|B^f(l)|\le d(l)\le1$,  $l$ can be colored with a color whose class is not full, a contradiction.  
Thus $|G|=sk$, for some integer $s\ge 2$. 
Now we may assume that $G$ is connected: adding an edge between the ends of two components will neither create a star, increase $\Delta(G)$ past $2$, nor make it easier to color $G$.

Choose a hub $u$ with body $U$ and tail $u'$ so that $d(u)$ is maximum; it exists by Lemma~\ref{hub}. 
If $|U|=k-1$ then set $G':=G[U+u-v]$, for some $v\in U$.  Else $|U|\le k-2$; set $G':=G[U+u]$. 
Anyway, $|G'|\le k-1$. 
Now $G'':=G-G'$ is not a star: in the first case, $v\in G''$ is isolated in $G''$; in the second case, if $c$ is the center of a star in $G''$ then {$G''$ is connected and} $d_{G''}(c)=k$, so $c\ne u'$ and $c$ is a hub in $G$ with body $N(c)-u'$ and tail $u'$,   contradicting the choice of $u$. 
 By Lemma~\ref{NS}, $G''$ has an end $l\ne u'$ such that $G_2:=G''-l$ is not a star. 
So $G_2$ has an SE $H$-coloring $f$. Put $G_1:={G[G'+l]}$. 
Now $|G_1|\le k$, and $f$ has at most $k-|G_1|$ large color classes. 
So the set of colors $H'(z)\subseteq H(z)$ whose classes in $f$ are not full with respect to $G$ will have size at least $|G_1|$ for all vertices $z\in G_1$. 
Set $H^*(z):=H'(z)\smallsetminus B^f(z)$. 
It suffices to show that $G_1$ has an injective $H^*$-coloring $g$, for then $f\cup g$ is an SE $H$-coloring of $G$, a contradiction.

Now $|H^*(l)|\ge|G_1|-1$, $|H^*(u)|\ge k-2\ge2$, and $|H^*(v)|\ge 
|G_1|$ for all $v\in U$. Fix $v\in U$. We may assume $\gamma \in H^*(v)\smallsetminus H^*(l)$. Choose $g(u)=\alpha$ so that $H^*(u,v)(\alpha)\ne \gamma$. Now extend $g$ to the independent set $U+l$ using \eqref{IND}.
\end{proof}
}

{
There are two derangements on $[3]$, which we denote by $D:=(1
~2~3)$ and $D':=(1~3~2)$. We denote by $D_1\dots D_{n-1}$ the $3$-cover $H$ of $v_1\dots v_n$ with $H(v_i,v_{i+1})=D_i$, where $D_i\in\{D,D'\}$.

\begin{cor}\label{path}
For all paths $P=v_1\dots v_t$ with $t\ne3$, if $k\ge3$ then $P$ is SEDP $k$-colorable.
\end{cor}

\begin{proof}

Let $H$ be a $k$-cover for $P$, and argue by induction on $t$.
Assume $k=3$ since otherwise we are done by Theorem~\ref{t13}. 
 By Proposition~\ref{Easy} and Theorem~\ref{n+d-1}(a),
$|G|> 3$. 
As in the proof of Theorem~\ref{t13}, we can assume that $|P|\bmod 3=0$, but we must take special care with the case $|G|=4<5$ since Lemma~\ref{NS} does not apply. 
As $v_1v_2+v_4$ is not a star, it has an SE $H$-coloring $f$. 
As $d(v_3)=2$, we can extend $f$ to an SE $H$-coloring of $P$.

We continue to follow the proof of Theorem~\ref{t13}, but again there is a problem with using Lemma~\ref{NS}, when $s=2$, because then $|G''|=4$. Thus we must treat this case separately. 
There is a third issue with calculating $H^*(u)$ because now $k=3$. But in the path case, $|U|=k-2$, so $\|u,G_2\|\le 1$. Thus we are left with the case $t=6$.

Now suppose $t=6$. If 
 $v_2v_3v_4$ or $v_3v_4v_5$ has an SE $H$-coloring, then  extend it as above. Else, without loss of generality, by Theorem~\ref{n+d-1},
 $H[v_2v_3v_4v_5]=DD'D$. 
 If $H(v_1,v_2)\ne D'\ne H(v_5,v_6)$ then by Theorem~\ref{n+d-1}, $v_1v_2v_3$ has an SE $H$-coloring $f$ and by Proposition~\ref{star}, $f$ can be extended to an SE $H$-coloring of $P$. 
 If $H(v_1,v_2)= D'= H(v_5,v_6)$ then $f(v_1,v_2,v_3,v_4,v_5,v_6):=(1,1,3,3,2,2)$ is an SE $H$-coloring of $P$. 
 Else, (say) $H(v_1,v_2)= D'$ and $H(v_5,v_6)\ne D'$.  
  If $H(v_5,v_6)(2)\ne 2$ then $f$ is still an SE $H$-coloring of $P$. Else $H(v_5,v_6)(1)\ne 2$ and $g(v_1,v_2,v_3,v_4,v_5,v_6):=(2,3,3,1,1,2)$ is an SE $H$-coloring of $P$.

\end{proof}

\section{General graphs with given maximum degree}\label{sec5}
In this section we show that all graphs with maximum degree at most $\Delta$ are SEDP $O(\Delta^{2})$-colorable.
It takes (us) a surprising amount of effort to prove this seemingly weak result.

\begin{thm}
\label{thm:thm:main}Let {$k\in\mathbb{Z}^+$}. All graphs $G=(V,E)$ with $3\Delta^{2}(G)\leq k$
are SEDP $k$-colorable. 
\end{thm}

\begin{proof}
Set $\Delta:=\Delta(G)$ and fix a k-cover $H$ for $G$ with $k\geq3\Delta^{2}$.
We argue by induction on $n$ that $G$ has an SEDP $H$-coloring if $|G|\leq n$.
If $n=0$  
then the result is trivial. The case  $n=k$ implies
the cases $1\leq n\leq k-1$: if $f$ is an SE $H$-coloring of $G$, then $f$ restricted
to $G'\subseteq G$ is an SE $H$-coloring, since all color classes of $f$ have
size at most one.

Otherwise  $n\geq k$ and the theorem holds for all graphs with fewer than $|G|-k$
vertices. As $G$ is $(\Delta+1)$-colorable
 and $n\geq 3\Delta^2$, $G$ has an { independent set
$X_{0}\subseteq V$ with $|X_0|:=\Delta$.} Set 
 {$X_{1}:=\bigcup\{N(v):v\in X_{0}\}$}.
Pick a $k$-set $X$ with $X_{0}\cup X_{1}\subseteq X\subseteq V$, 
and set $X_{2}:=X\smallsetminus (X_{0}\cup X_{1})$.
By induction, $G-X$ has a SE $H$-coloring $f_{0}$. Notice that $E(X_{0},V\smallsetminus X_{1})=\emptyset$.

It suffices to extend $f_0$ to an $H$-coloring $f$ that is injective on $X$.
We will construct $f$ iteratively while maintaining two invariants:
\begin{enumerate}[label=(\Alph{enumi})]
\item \label{enu:A_0}If $x\in X_{0}$ is uncolored then any color in $B(x)$ has already
been used on $X$.
\item \label{enu:A_2}If $z\in X_{2}$ is uncolored then all colors used on $X_{0}$ are
in $B(z)\cup(\Gamma\smallsetminus H(z))$.
\end{enumerate}
Initially, $X$ is uncolored, and $f:=f_{0}$. As $N[x]\subseteq X$
for all $x\in X_{0}$, both \ref{enu:A_0} and \ref{enu:A_2} hold.

We will maintain dynamic variables $D,\overline{D},R,A(x),B(x)$ that depend on the
current value of $f$. 
{ Sometimes we will specify this current value by writing, for example, $\overline{D}^f$.} 
Suppose that at a given step $f\supseteq f_{0}$ is  an injective
$H$-coloring. Let $D$ be the domain of $f$ restricted to $X$, and let
$R$ be the  range $f(D)$ of $f$ restricted to $D$. Set $\overline{D}:=X\smallsetminus D$.
So $\overline{D}$ is the set of uncolored vertices. Also, for $i\in\{0,1,2\}$,
we have dynamic variables $D_{i}:=D\cap X_{i}$, $\overline{D}_{i}:=\overline{D}\cap X_{i}$,
and $R_{i}:=f(D_{i})$.  As $f$ is injective, $|D|=|R|$ and $|D_{i}|=|R_{i}|.$ When
$f$ is extended to $x\in\overline{D}$, the color $f(x)$ must be in  $H(x)$ and
neither blocked from $x$ nor used on $X$;  let $A(x):=H(x)\smallsetminus(B(x)\cup R)$
be this set of \emph{available} colors for $x$.

Now we rewrite \ref{enu:A_0} and \ref{enu:A_2} as:
\addtocounter{equation}{1}
\begin{equation}
\tag*{(\theequation)}
\text{(a) }B(x)\subseteq R\text{ for all }x\in\overline{D}_{0}~\text{ and }~\text{(b) }R_{0}\subseteq B(z)\cup(\Gamma\smallsetminus H(z))\text{ for all }z\in\overline{D}_{2}.\label{eq:inv}
\end{equation}
Since $\|X_{0},X_{1}\|\leq\Delta| X_{0}|=\Delta^{2}$ and $|X_{2}|=k-|X_{0}|-|X_{1}|\geq2\Delta^{2}-\Delta$,
initially,
\begin{equation}
\text{if }~\overline{D}_{1}\ne\emptyset~\text{ then }~|\overline{D}_{2}|\geq\|\overline{D}_{0},\overline{D}_{1}\|+(\Delta-2)|\overline{D}_{0}|+\Delta.\label{eq:inv2}
\end{equation}

Call an  $H$-coloring $f\supseteq f_0$ \emph{promising} if it is injective on $X$ and  satisfies \ref{eq:inv}
and \eqref{eq:inv2}. Now~\eqref{eq:inv2} implies
\begin{equation}
\text{if }~\overline{D}_{1}\ne\emptyset~\text{ then }~|\overline{D}_{2}|\geq\Delta~~\mbox{for every promising  \ensuremath{H}-coloring.}\label{eq:end2}
\end{equation}
Arguing by induction on $|D|$, it suffices to show that every promising $H$-coloring $f$ can be extended to a promising $H$-coloring with larger domain.

\smallskip
\emph{Case 1:} $\|\overline{D}_{0},\overline{D}_{1}\|>0.$ Say $y\in\overline{D}_{1}$
with $\|\overline{D}_{0},y\|>0$. Using \eqref{eq:end2}, $|\overline{D}_1|+|\overline{D}_2|\ge \Delta+1$, so there is $\beta\in A(y)$.
For $x\in\overline{D}_{0}\cap N(y)$, let $\alpha_{x}:=H(y,x)(\beta)$. Set
\[W:=\{x\in \overline{D}_0:\alpha_x\in A(x)-\beta\}~\mathrm{and}~{\Lambda}:=\{\alpha_x:x\in W\}.\]

 Let {$\Lambda_{0}\subseteq \Lambda$} be maximum subject to there being an injection $\iota: \Lambda_0\to\overline{D}_{2}\smallsetminus N(y)$ such that 
 \[\mbox{

 $g' :=f|D\cup \iota^{-1}$ 
  is an injection and $f\cup g'$ is an $H$-coloring.}\]
Put $W_0:=\{x\in W:\alpha_x \in \Lambda_0\}$ and  $Z:=\ra(\iota)$ and $g:=f\cup g'$.

If $W_{0}=W$ then set  $h:=g+(y,\beta)$. We claim that $h$ is promising. As the newly colored vertices in $Z$
are neither colored with $\beta$ nor adjacent to $y$, $g'+(y,\beta)$ is an injection and $h$ is an $H$-coloring. 
Moreover, \ref{eq:inv}(a)
holds since for every $x\in W=W_0$  there is a vertex $z\in Z$ with
$g(z)=\alpha_x$, and \ref{eq:inv}(b) holds trivially. Since the number of colored vertices
in $X_{2}$ increases by at most the number of uncolored neighbors of $y$ in $X_{0}$,
\eqref{eq:inv2} holds.

Otherwise \emph{ }$W_{0}\ne W$. Let $x\in W\smallsetminus W_{0}$. 
Set $Z'=N(y)\cap \overline{D}_{2}$. 
Note that $\|x,Z\|=\emptyset$. Since $\alpha_x\notin \Lambda_0$,  $g'(z)\ne \alpha_x$ 
for all  $z\in Z$.
 By the maximality of $\Lambda_0$, (*) $\alpha_x\notin A^{g}(z)$ for any $z\in \overline{D}_2\smallsetminus (Z\cup Z')$. Thus $h':=g'+(x,\alpha_x)$ is an injection and $f\cup h'$ is an $H$-coloring. 
Let $h'':D\cup Z'\cup Z''+x\to\Lambda$ be an injection such that $h'\subseteq h''$, $h''(x)=\alpha_x$, and $h:=f\cup h''$ is an $H$-coloring. 
This is possible since 
 $\|\overline{D}_{0},\overline{D}_{2}\|=0$ and $|A(y)|\ge 2\Delta-1$ for  all $y\in Z'$ by \ref{eq:inv}(b). 
We claim that $h$ is promising. Trivially, \ref{eq:inv}(a) holds, and \ref{eq:inv}(b) holds by (*). 
As $|\overline{D}_2^h|=|\overline{D}_2|-|Z\cup Z'|\ge |\overline{D}_2|-\Delta+1$, $\|\overline{D}_0,\overline{D}_1\|\ge \|\overline{D}^h_0,\overline{D}^h_1\|+1$, and 
$|\overline{D}_0|\ge|\overline{D}_0^h|+1$,
\begin{align}\notag
  |\overline{D}^h_{2}|&\ge |\overline{D}_{2}|-\Delta+1
\geq_{\eqref{eq:inv2}}\|\overline{D}_{0},\overline{D}_{1}\|+(\Delta-2)|\overline{D}_{0}|+\Delta-\Delta+1 \\
&\ge (\|\overline{D}^h_{0},\overline{D}^h_{1}\|+1)+(\Delta-2)(|\overline{D}^h_{0}|+1)+1\ge
\|\overline{D}^h_{0},\overline{D}^h_{1}\|+(\Delta-2)|\overline{D}^h_{0}|+\Delta,\notag
\end{align} 
so \eqref{eq:inv2} holds for $h$, and $h$ is promising.

\smallskip
\emph{Case 2:} $\|\overline{D}_{0},\overline{D}_{1}\|=0.$
Now the vertices of $\overline{D}_0$ are isolated in $G':=G[\overline{D}]$.
We will color $\overline{D}$ in the order  $\overline{D}_1,\overline{D}_2,\overline{D}_0$, so 
 \ref{eq:inv} and \eqref{eq:inv2} will continue to hold.
If there is $y\in\overline{D}_{1}$, then 
\[
|A(y)|\geq (k-|D|)-|B(y)|\geq|\overline{D}|-\Delta\geq|\overline{D}_{1}|+|\overline{D}_{2}|-\Delta\geq_\eqref{eq:end2} 1,
\]
so we can color $y$ from $A(y)$ and keep \ref{eq:inv}
and \eqref{eq:inv2}.

So assume $X_1\subseteq D$, but there is a vertex $z\in\overline{D}_{2}$. By \ref{eq:inv}(b), 
\begin{equation}
H(z)\cap R_{0}\subseteq H(z)\cap B(z).\label{eq:H(z)}
\end{equation}
 As $|X_{0}|=\Delta$, the number of available colors for $z$ is 
\begin{align*}
|A(z)| & \geq|H(z)|-|H(z)\cap((B(z)\cup R_{0})\cup(R_{1}\cup R_{2}))|\geq_{\text{\eqref{eq:H(z)}}}k-|B(z)|-|R|+|R_{0}|\\
 & \geq|\overline{D}|+|R_{0}|-\Delta\geq|\overline{D}_{2}|+|\overline{D}_{1}|+|\overline{D}_{0}|+|D_{0}|-\Delta\geq|\overline{D}_{2}|+|X_{0}|-\Delta\geq|\overline{D}_{2}|.
\end{align*}
Color $z$ from $A(z)$.  Now \ref{eq:inv} holds since $D_0$ did not change, and \eqref{eq:inv2} holds since $\overline{D}_1=\emptyset$.

Finally, assume $X_1 \cup X_2\subseteq D$, but there is   $x\in\overline{D}_{0}$.
The number of available colors for $x$ is
\begin{align*}
|A(x)| & \geq|H(x)|-|B(x)\cup R|\geq_{\ref{eq:inv}(a)}k-|R|\geq|X|-|D|=|\overline{D}_{0}|.
\end{align*}
Color $x$ from $A(x)$.  Now \ref{eq:inv}(a) holds since $x\notin N(D_0)$, \ref{eq:inv}(b) holds since $\overline{D}_2=\emptyset$, and \eqref{eq:inv2} holds since $\overline{D}_1=\emptyset$.
\end{proof}


\color{black}
\section{Concluding remarks}

1. It seems that EDP  coloring and SEDP coloring are natural extensions of EL coloring and SEL coloring with interesting properties.

\medskip
2. Let $g(D)$ (respectively, $g_{se}(D)$) be the minimum $k_0$ such that for every $k\geq k_0$ every graph with maximum degree at most $D$ is EDP (respectively, SEDP)
$k$-colorable. 
It would be interesting to determine whether  $g(D)$ and/or
$g_{se}(D)$ are  $O(D)$.
 Our Theorem~\ref{thm:thm:main} gives only a quadratic upper bound on $g_{se}(D)$. We even do not know $g(3)$. 
 Examples~\ref{ex:4} and~\ref{ex:5} show that $g(3)\geq 7$ and Theorem~\ref{thm:thm:main} shows that $g(3)\leq 27$.

\medskip
3. It looks that the differences between the bounds for EDP  coloring and EL coloring are more significant when the graphs are dense. In particular, we think that every $n$-vertex graph with maximum degree less than $n/2$ is EDP $n$-colorable, but cannot prove it.
\medskip

4. 
As observed by Dvo\v rak and Postle~\cite[Observation 3]{2018DvPo}, in DP $k$-coloring it is enough to consider plain $k$-covers.
In all cases that we have come across, if a graph $G$   with a $k$-cover $H$ is not
SE $H$-colorable, then  also there is a plain $k$-cover $H'$ of $G$ such that $G$
is  not
SE $H'$-colorable. It would be interesting to determine whether this  always holds for SEDP colorings. This is not the case for SEL coloring, see~\cite[Example 22]{KKX4} (Also arXiv:2504.14711).

\bibliographystyle{plain}
\bibliography{Ref}

\end{document}